\documentclass[reqno]{amsart}

\usepackage{amsmath,amssymb,amsthm, tikz}

\newtheorem*{proposition}{Proposition}
\newtheorem*{thm}{Theorem}

\newtheorem*{lemma}{Lemma}

\begin{document}

\title[]{Bad Science Matrices}

\author[]{Stefan Steinerberger}

\address{Department of Mathematics, University of Washington, Seattle, WA 98195, USA}
 \email{steinerb@uw.edu}

\thanks{The author is partially supported by the NSF (DMS-2123224).}

\begin{abstract} Inspired by the bad scientist who keeps repeating an experiment 20 times to get a single outcome with $p < 0.05$, we consider matrices $A \in \mathbb{R}^{n \times n}$ whose rows are normalized in $\ell^2$ and for which $2^{-n}\sum_{x \in \left\{-1,1\right\}^n} \|Ax\|_{\ell^{\infty}}$ is large. They correspond to affine transformations of the discrete unit cube to points with, on average, at least one large coordinate. Such matrices can be seen as a collection of fair tests on a fair coin where at least one outcome is typically atypical. We prove that, as $n \rightarrow \infty$, the quantity can scale as
 $$  \max_{A \in \mathbb{R}^{n \times n}} \frac{1}{2^{n}}\sum_{x \in \left\{-1,1\right\}^n} \|Ax\|_{\ell^{\infty}} = (1+o(1)) \cdot \sqrt{2\log{n}}.$$
We also present candidate maximizers up to dimension $n \leq 8$ which appear to be highly structured and have nice closed-form solutions.
\end{abstract}
\maketitle

\maketitle

\section{Introduction and Result}
 \subsection{The dishonest scientist.}
Suppose we are given a (fair) coin which is going to be tossed $n$ times resulting in $x_1, \dots, x_n \in \left\{-1,1\right\}$. We are interested in finding out whether the coin is biased or unbiased. A natural first step would be to recall that if the coin was fair, then, as $n$ gets large, we have
$$ \frac{x_1 + x_2 + \dots + x_n}{\sqrt{n}} \implies \mathcal{N}(0,1).$$
More generally, if $a \in \mathbb{R}^n$ is any vector with unit length, $\|a\|_{\ell^2} = 1$, then the random inner product $ X_{a} = \left\langle a, (x_1, x_2, \dots, x_n) \right\rangle$ has expectation
$ \mathbb{E}X_{a} = 0$  and $\mathbb{V}X_{a}  = 1.$
Moreover, $X_{a}$ is as unlikely to yield a large number as the average is: it follows from Hoeffding's inequality \cite[Theorem 2.2.2]{vershynin} that, for any $t>0$,
$$ \mathbb{P}\left( |X_{a}| \geq t\right) \leq 2 e^{-t^2/2}.$$
In particular, if we fix a vector $a \in \mathbb{R}^n$ and then observe $n$ coin tosses via the random variable $X_a$, then we may conclude that $X_a$ being large is indeed indicative of a bias in the coin. We could now imagine a dishonest scientist going ahead and selecting $n$ unit vectors $a_1, a_2, \dots, a_n$ a priori and then evaluating the random coin toss by looking at the random variables $X_{a_1}, X_{a_2}, \dots, X_{a_n}$. If any of these random variables, say $X_{a_i}$, is large, then the dishonest scientist obtains a small $p$ value, concludes that the coin is biased and can publish the (incorrect) result: after all, the test $X_{a_i}$ was fair.
The question is now: how dishonest can the scientist be?\\

 Using the matrix $A \in \mathbb{R}^{n \times n}$ whose rows are given by the vectors $a_1, \dots, a_n$ (which we always assume to be normalized in $\ell^2$), we measure the dishonesty via
$$ \beta(A) = \frac{1}{2^n} \sum_{x \in \left\{-1,1\right\}^n} \|Ax\|_{\ell^{\infty}}.$$
This corresponds to the average largest value that is attained by this dishonest testing mechanism for a random sample of coin tosses. Note that $\beta(\mbox{Id}_{n \times n}) = 1$. One could argue that the most `honest' matrix is a rescaling of the all 1's matrix
$$ A = \frac{1}{\sqrt{n}} \mathbf{1} \qquad \mbox{for which} \qquad \beta(A) \sim \sqrt{\frac{2}{\pi}} = 0.79788\dots$$
corresponding to simply averaging the result and accepting that outcome without running any further tests. Regarding the smallest value $\beta$ can attain, it seems that $\beta(A) \geq 1/\sqrt{2}$ is a plausible guess; we have not pursued this. Instead we ask: how large can $\beta(A)$ be? Note that the question is also of obvious interest from a purely geometric perspective (measuring images of the discrete cube in $\ell^{\infty}$).

\subsection{Results} We start with a simple heuristic that points to the right scale.
Taking the matrix $A$ to be orthogonal, any two random variables $\left\langle a_i, x \right\rangle$ and $\left\langle a_j, x \right\rangle$ are `almost' independent. We could thus optimistically argue that we get essentially $n$ nearly independent and nearly standard Gaussian variables. The maximum of $n$ independent standard Gaussian random variables is $\sim \sqrt{2 \log{n}}$. 

\begin{thm}
Among matrices with rows satisfying $\|a_i\|_{\ell^2} \leq 1$, as $n \rightarrow \infty$,
$$\max_{A \in \mathbb{R}^{n \times n}} \frac{1}{2^{n}}\sum_{x \in \left\{-1,1\right\}^n} \|Ax\|_{\ell^{\infty}} =(1+o(1)) \cdot \sqrt{2\log{n}}.$$
\end{thm}

The proof shows that random matrices with iid entries $\pm 1/\sqrt{n}$ attain this asymptotic rate. The result says very little for specific values of $n$ and, at least numerically, extremal matrices seem to exhibit quite a bit more structure than generic random matrices. Needless to say, this difference is presumably subtle and hiding in the $o(1)$ term. One could argue that the smoothing coming from the central limit theorem determines the leading order and the real problem is hiding in the lower order terms. It would be interesting to have a more precise asymptotic expansion, this might be quite difficult. A construction of Kunisky \cite{kun} can be adapted to yield an asymptotically suboptimal but completely explicit family of matrices.
 \begin{proposition} If $n$ is a power of 2, there exists an explicit matrix (the transpose of the discrete Haar transform) such that $\|Ax\|_{\ell^{\infty}}$ is independent of $x \in \left\{-1,1\right\}^n$ and
 $$  \frac{1}{2^{n}}\sum_{x \in \left\{-1,1\right\}^n} \|Ax\|_{\ell^{\infty}} = \sqrt{\log_2{n} + 1} \sim  \sqrt{1.442\log{n}}.$$
 \end{proposition} 
 This construction leads to the currently best known results for $n=2,4,8$ (see \S 2); starting at $n\geq 128$ it seems that random $\pm 1/\sqrt{n}$ matrices will typically have a larger value of $\beta$.
The existence of such constructions does suggest that one might be hopeful about further connections with the geometry of number, special lattices, coding theory as well as combinatorial aspects of design theory. One reason why we believe the problem to be interesting is that extremal matrices seem to have an interesting structure, see \S 2 for a discussion up to $n \leq 8$.
\begin{quote}
\textbf{Problem.} Which $A \in \mathbb{R}^{n \times n}$ maximize $\sum_{x \in \left\{-1,1\right\}^n} \|Ax\|_{\ell^{\infty}}$~?\\
\end{quote}

Already for relatively small values of $n$, the question appears quite nontrivial. Even just approaching it numerically seems challenging as the number of variables grows quadratically in the dimension. Moreover, with a Central Limit Theorem smoothing in each coordinate, the inverse question is naturally hard. Finally, when working numerically, there appear to be a large number of local maxima that are not global. Some of these seem to be attracting numerical methods; there are matrices in different dimensions that are not optimal but showed up more frequently in the experiments than matrices that perform better (see below).

\begin{figure}[h!]
\begin{center}
\boxed{
\begin{tabular}{ l c c c c c c c }
 $ n$  & 2 & 3  & 4 & 5 & 6 & 7 & 8 \\
 $ \beta_n \geq $ & 1.41 & 1.57 & 1.73 & 1.79 & 1.86 & 1.93 & 2 \\
\end{tabular}
}
\end{center}
\caption{Lower bounds for $\beta_n = \max \beta(A)$ when $n \leq 8$, see \S 2.}
\end{figure}

\subsection{Related results.}  We are not aware of any directly related results.  There are many results surrounding vector balancing problems \cite{bana2, bana, haj, kun, li, reis, spencer}. There, one is given vectors $x_1, \dots, x_m$ in some space $X$ and is asked to pick signs so that $\| \pm x_1 \pm x_2 \pm \dots \pm x_m\|_X$ is small. This is somewhat dual as one normalizes the columns of the matrix (as opposed to the rows): vector balancing instead of functional balancing. Vector balancing problems are more difficult than our problem because they ask for a minimum over all signs while we take an expectation. The most famous problem in this direction is perhaps the K\'omlos conjecture: given $v_1, \dots, v_n \in \mathbb{R}^n$, all normalized in $\ell^2$, are there signs with $\| \pm v_1 \pm v_2 \pm \dots \pm v_n\|_{\ell^{\infty}} \leq K$ for some $K$ independent of $n$? The result is known to be true when $K$ is replaced by $c\sqrt{\log{n}}$, see \cite{bana3}. One interesting aspect of the story is that it is not easy to get nontrivial lower bounds on $K$: the best result is due to Kunisky \cite{kun} who proved $K \geq \sqrt{2} + 1$. The construction of Kunisky appears to be either optimal or near-optimal for our problem when $n=2,4,8$ (and provably non-optimal in high dimensions). This raises the natural question whether extremal configurations for our problem might inspire good examples for the Koml\'os conjecture (or might do so after a small modification): having a large average does not imply the minimum is large but a frequently attained small minimum is not good for the average either.\\
The problem has a number of geometric interpretations. Restricting to orthogonal matrices $Q \in \mathbb{R}^{n \times n}$, we are reminded that random rotations preserve the $\ell^2-$norm but that the largest entry in a random vector on the ball gains a logarithmic factor. Another way of thinking about our problem is that we want to place thickened hyperplanes
$ H_{i} = \left\{ x \in \mathbb{R}^n: \left|\left\langle x, a_i \right\rangle\right| \leq \beta(A) \right\}$
in such a way that as many of the points in the discrete cube as possible are far away from at least one hyperplane. Covering lattice points of the discrete cube with hyperplanes is a classic subject, see for example Alon-F\"uredi \cite{alon}. There is a beautiful question by Henk-Tsintsifas \cite{henk} (maybe more related in spirit but it should be more widely known): what is the largest constant $\gamma_n$ such that any translation and rotation of the scaled unit cube $\gamma_n [0,1]^d$ always intersects the lattice $\mathbb{Z}^d$?  Henk-Tsintsifas \cite{henk} show that $\gamma_2  = \gamma_3 = \sqrt{2}$.  It is known \cite{bana} that $\gamma_n \lesssim \sqrt{\log{n}}$. The Koml\'os conjecture would imply $\gamma_n \leq C$ independently of $n$.

\section{Small values of $n$}

\subsection{Summary.} The purpose of this section is to illustrate various interesting configurations that we found for small values of $n$. We tried a number of different methods, the most effective (for small $n$) was as follows:
\begin{enumerate}
\item Sample a number of random orthogonal matrices, compute their $\beta$ value.
\item Use the orthogonal matrix $O$ with the largest $\beta$ to start $M_0 = O$.
\item Add a small multiple of a random Gaussian matrix $G$
$$ M_n + \varepsilon G, ~\mbox{then renormalize the rows to obtain}~M_{n+1}$$
and see whether $\beta(M_{n+1}) > \beta(M_n)$. If not, then $M_{n+1} = M_n$.
\item Once the matrix stops being updated, decrease the step-size $\varepsilon$.
\item See whether the elements in $M_n$ resemble any known numbers.
\end{enumerate}

This approach does not scale particularly well.  On the other hand, it is easy to run
many times on different random initializations. The purpose of this section is to report on some of the configurations that were discovered in the process. We have not tried to prove
that any of our potentially optimal configurations are optimal and this does not appear to be easy. While easy
for $n=2$ and possibly doable for $n=3$ with enough patience, dealing with the sum of the maximum of $2^n$ vectors appears to be difficult. 

\subsection{The case $n=2$} 
The worst case for $n=2$ is given by 
$$ A = \frac{1}{\sqrt{2}} \begin{pmatrix} 1 &  1 \\ 1 & - 1 \end{pmatrix} $$
which rotates each of the four point $(\pm 1, \pm 1)$ into a point with coordinate of size at least $\sqrt{2}$. It is not difficult to see that this is optimal: since any extremal matrix has to have rows of $\ell^2-$norm 1, there is the natural ansatz
$$ A = \begin{pmatrix} \cos{(s)} &  \sin{(s)} \\ \cos{(t)} & - \sin{(t)} \end{pmatrix} $$
Then the sum $\sum_{x} \|Ax\|_{\ell^{\infty}}$ involves the maximum of 4 trigonometric functions which is bounded by $\sqrt{2}$. There are presumably a number of elementary geometric arguments. We notes 
$$  \max_{A \in \mathbb{R}^{2 \times 2}} \beta(A) = \sqrt{2} \sim 1.4142\dots$$

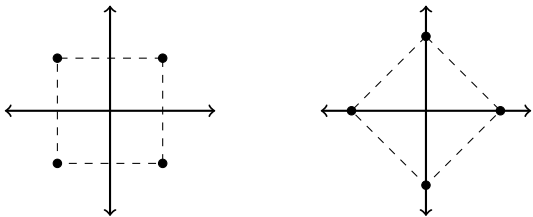
\begin{figure}[h!]
\begin{tikzpicture}[scale=0.7]
\draw [thick, <->] (-2, 0) -- (2, 0);
\draw [thick, <->] (0, -2) -- (0, 2);
\filldraw (1,1) circle (0.08cm);
\filldraw (-1,1) circle (0.08cm);
\filldraw (1,-1) circle (0.08cm);
\filldraw (-1,-1) circle (0.08cm);
\draw[dashed] (1,1) -- (-1,1) -- (-1,-1) -- (1,-1) -- (1,1) ;
\draw [thick, <->] (6-2, 0) -- (6+2, 0);
\draw [thick, <->] (6, -2) -- (6, 2);
\filldraw (6, 1.4142) circle (0.08cm);
\filldraw (6, -1.4142) circle (0.08cm);
\filldraw (6 + 1.4142, 0) circle (0.08cm);
\filldraw (6 - 1.4142, 0) circle (0.08cm);
\draw[dashed] (6, 1.4142) -- (6-1.4142,0) -- (6, -1.4142) -- (6+1.4142, 0) -- (6, 1.4142);
\end{tikzpicture}
\caption{The extremal case for $n=2$.}
\end{figure}

\subsection{The case $n=3$} This case already appears somewhat nontrivial. The best matrix we found is
$$ A = \frac{1}{2} \begin{pmatrix} -1 & -1& \sqrt{2} \\  -\sqrt{2} & 0 & \sqrt{2}  \\ 1& 1 & \sqrt{2} \end{pmatrix}.$$

It maps the 8 vertices $\left\{-1,1\right\}^3$ to points whose largest coordinate in absolute value is either $\sqrt{2}$ or $\sqrt{3}$ (half the time each). This shows
$$  \max_{A \in \mathbb{R}^{3 \times 3}} \beta(A) \geq \frac{ \sqrt{2} + \sqrt{3}}{2} \sim 1.573\dots.$$
We note that this matrix is \textit{not} orthogonal. This is interesting insofar as it shows that it may be optimal to have some built-in dependence to get the most biased results. The best orthogonal matrix we found was relatively close in terms of $\beta$
$$ Q = \begin{pmatrix} \frac{1}{2} & \frac12 & \frac{1}{\sqrt{2}} \\  -\frac{1}{2} & -\frac{1}{2} & \frac{1}{\sqrt{2}}  \\ \frac{1}{\sqrt{2}} & \frac{1}{\sqrt{2}} & 0 \end{pmatrix} \qquad \mbox{with} \qquad \beta(Q) = \frac{2 + 3 \sqrt{2}}{4} \sim 1.560\dots$$
\begin{figure}[h!]
\begin{tikzpicture}[scale=0.7]
\node at (0,0) {\includegraphics[width=0.3\textwidth]{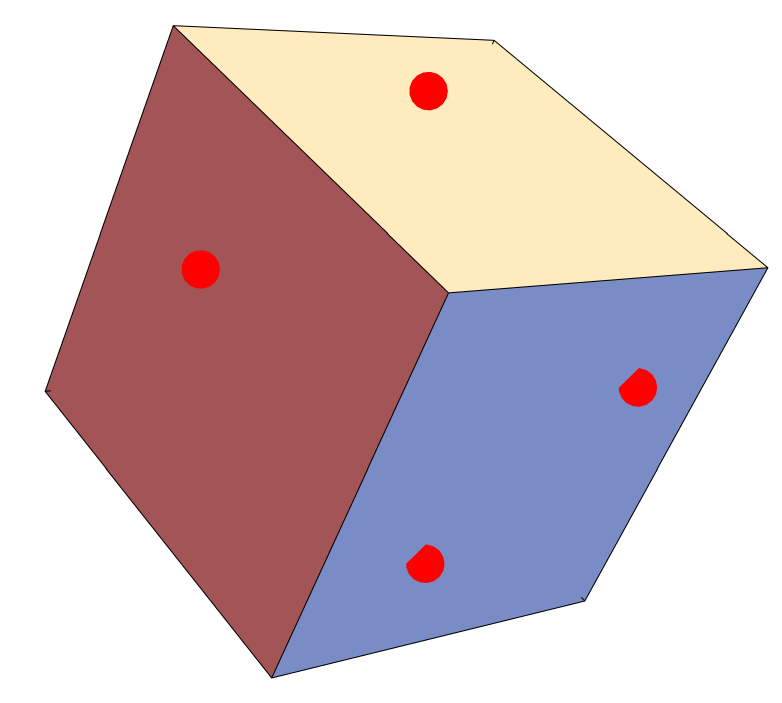}};
\end{tikzpicture}
\caption{Four of the eight points $Ax$ in relation to a cube centered in the origin of sidelength 2.7.}
\end{figure}

\subsection{The case $n=4$} 
The best example we found for $n=4$ is 
$$ A = \frac{1}{\sqrt{3}} \begin{pmatrix} 0 & 1 & -1 & 1 \\ 1 & -1 & 0 & 1 \\ -1 & -1 & -1 & 0 \\ -1 & 0 & 1 & 1 \end{pmatrix} $$
which shows that
$$  \max_{A \in \mathbb{R}^{4 \times 4}} \beta(A) \geq \sqrt{3} \sim 1.732\dots.$$
This matrix is orthogonal and it sends every vertex of the cube to a point with $\ell^{\infty}-$norm exactly $\sqrt{3}$. To make things more amusing, the solution is not unique. The matrix
$$ B = \frac{1}{\sqrt{3}} \begin{pmatrix} 1 & 1 & 1 & 0 \\ 1 & -1 & -1 & 0 \\ 1 & -1 & 1 & 0 \\ 1 & 1 & -1 & 0 \end{pmatrix} \qquad \mbox{also satisfies} \quad \beta(B) = \sqrt{3}$$
is not orthogonal and completely ignores the fourth coordinate: it maps $\left\{-1,1\right\}^4$ into a $3-$dimensional subspace. `Good' configurations ignoring an entire coordinate seems to be a common feature and seem to appear in other dimensions as well. Kunisky's construction (see the Proof of the Proposition for details) provides yet another example
$$C = \frac{1}{\sqrt{3}} \left(
\begin{array}{cccc}
 1 & 1 & 1 & 0 \\
 1 & 1 & -1 & 0 \\
 1 & -1 & 0 & 1 \\
 1 & -1 & 0 & -1 \\
\end{array}
\right) \quad \mbox{also satisfying} \quad \beta(C) = \sqrt{3}.$$

\subsection{The case $n=5$} The best example we found for $n=5$ is the matrix
$$ A = \frac{1}{2 \sqrt{3}} \left(
\begin{array}{ccccc}
 2 & 2 & 0 & 0 & 2 \\
 -2 & 2 & 0 & 2 & 0 \\
 -2 & 0 & 0 & -2 & 2 \\
 0 & -\sqrt{3} & \sqrt{3} & \sqrt{3} & \sqrt{3} \\
 0 & \sqrt{3} & \sqrt{3} & -\sqrt{3} & -\sqrt{3} \\
\end{array}
\right) $$
which shows that
$$  \max_{A \in \mathbb{R}^{5 \times 5}} \beta(A) \geq \frac{2+ 3\sqrt{3}}{4} \sim 1.799\dots.$$
It sends the 32 vertices to 8 vertices with maximal coordinate 2 and 24 vertices with maximal coordinate $\sqrt{3}$. Again, the matrix is not unique and 
$$
B = \frac{1}{2\sqrt{3}}\left(
\begin{array}{ccccc}
 -\sqrt{3} & -\sqrt{3} & 0 & \sqrt{3} & -\sqrt{3} \\
 2 & -2 & 0 & 0 & 2 \\
 \sqrt{3} & \sqrt{3} & 0 & -\sqrt{3} & -\sqrt{3} \\
 -2 & 0 & 0 & -2 & 2 \\
 0 & 2 & 0 & 2 & 2 \\
\end{array}
\right)
$$
has the exact same properties while also completely ignoring the third coin toss. There are also other types of near-maximizers. A sub-optimal matrix that keeps persistently showing in numerical experiments is
$$ C = \left(
\begin{array}{ccccc}
 \frac{1}{\sqrt{3}} & 0 & -\frac{1}{\sqrt{3}} & \frac{1}{\sqrt{3}} & 0 \\
 -\frac{1}{\sqrt{29}} & -\frac{3}{\sqrt{29}} & -\frac{3}{\sqrt{29}} & -\frac{3}{\sqrt{29}} & -\frac{1}{\sqrt{29}} \\
 -\frac{1}{\sqrt{3}} & 0 & 0 & \frac{1}{\sqrt{3}} & -\frac{1}{\sqrt{3}} \\
 \frac{1}{\sqrt{29}} & \frac{3}{\sqrt{29}} & -\frac{1}{\sqrt{29}} & -\frac{3}{\sqrt{29}} & -\frac{3}{\sqrt{29}} \\
 \frac{1}{2} & -\frac{1}{2} & \frac{1}{2} & 0 & -\frac{1}{2} \\
\end{array}
\right)$$
with $\beta(C) = 1.789$.

\subsection{The case $n=6$} 
The best example we found for $n=6$ is the matrix
$$ A = \frac{1}{2 \sqrt{3}} \left(
\begin{array}{cccccc}
 -\sqrt{3} & 0 & \sqrt{3} & -\sqrt{3} & 0 & -\sqrt{3} \\
 \sqrt{3} & 0 & -\sqrt{3} & \sqrt{3} & 0 & -\sqrt{3} \\
 \sqrt{3} & 0 & \sqrt{3} & \sqrt{3} & 0 & \sqrt{3} \\
 2 & -2 & 0 & -2 & 0 & 0 \\
 -\sqrt{3} & 0 & -\sqrt{3} & -\sqrt{3} & 0 & \sqrt{3} \\
 2 & 2 & 0 & -2 & 0 & 0 \\
\end{array}
\right)$$
showing that
$$  \max_{A \in \mathbb{R}^{6 \times 6}} \beta(A) \geq  \frac{\sqrt{3} + 2}{2} \sim 1.866\dots.$$
Much as in the case $n=5$, there appears to be a matrix that keeps popping up in numerical simulations that is very nearly as good and, much like when $n=5$, contains several entries featuring $\sqrt{29}$
 $$ B = \left(
\begin{array}{cccccc}
 \frac{1}{\sqrt{29}} & -\frac{3}{\sqrt{29}} & 0 & \frac{3}{\sqrt{29}} & \frac{1}{\sqrt{29}} & -\frac{3}{\sqrt{29}} \\
 \frac{1}{2} & 0 & \frac{1}{2} & 0 & \frac{1}{2} & \frac{1}{2} \\
 \frac{3}{\sqrt{29}} & \frac{3}{\sqrt{29}} & 0 & \frac{1}{\sqrt{29}} & -\frac{1}{\sqrt{29}} & -\frac{3}{\sqrt{29}} \\
 0 & 0 & 0 & -\frac{1}{\sqrt{3}} & \frac{1}{\sqrt{3}} & -\frac{1}{\sqrt{3}} \\
 \frac{1}{2} & -\frac{1}{2} & 0 & -\frac{1}{2} & -\frac{1}{2} & 0 \\
 -\frac{1}{2} & 0 & \frac{1}{2} & 0 & -\frac{1}{2} & -\frac{1}{2} \\
\end{array}
\right)$$
with $\beta(B) =  \left(6+2 \sqrt{3}+\sqrt{29}\right)/8 = 1.856\dots$. It is not clear to us whether this might be more than a curious coincidence. One might get the impression that some of these rows serve as convenient building blocks
that are then combined with other building blocks in other dimensions; it would be nice to understand this better.

\subsection{The case $n=7$} 
The best example we found for $n=7$ is the matrix
$$ A = \frac{1}{2} \left(
\begin{array}{ccccccc}
 0 & -1 & 1 & 0 & -1 & 0 & -1 \\
 0 & -1 & -1 & 0 & 1 & 0 & 1 \\
 0 & 0 & 1 & 0 & 1 & 1 & -1 \\
 0 & 0 & -1 & -1 & -1 & 1 & 0 \\
 0 & 0 & 1 & -1 & 1 & -1 & 0 \\
 0 & 0 & 2/\sqrt{3} & 0 & 0 & 2/\sqrt{3} & 2/\sqrt{3} \\
 0 & 0 & 1 & 0 & -1 & -1 & 1 \\
\end{array}
\right)$$
showing that
$$  \max_{A \in \mathbb{R}^{7 \times 7}} \beta(A) \geq  \frac{\sqrt{3} + 6}{4} \sim 1.933\dots.$$
Note that, again, an entire column is 0: the first coin toss is completely ignored.

\subsection{The case $n=8$} 
The best example for $n=8$ that we found is one considered by Kunisky \cite{kun}, the transpose of the matrix corresponding to the discrete Haar wavelet decomposition. This matrix has also recently proven to be useful in related work of Li \& Nikolov \cite{li}. It is given by

$$A= \frac{1}{2}\left(
\begin{array}{cccccccc}
 1 & 1 & 1 & 0 & 1 & 0 & 0 & 0 \\
 1 & 1 & 1 & 0 & -1 & 0 & 0 & 0 \\
 1 & 1 & -1 & 0 & 0 & 1 & 0 & 0 \\
 1 & 1 & -1 & 0 & 0 & -1 & 0 & 0 \\
 1 & -1 & 0 & 1 & 0 & 0 & 1 & 0 \\
 1 & -1 & 0 & 1 & 0 & 0 & -1 & 0 \\
 1 & -1 & 0 & -1 & 0 & 0 & 0 & 1 \\
 1 & -1 & 0 & -1 & 0 & 0 & 0 & -1 \\
\end{array}
\right)
$$

This matrix shows that
$$  \max_{A \in \mathbb{R}^{8 \times 8}} \beta(A) \geq 2.$$
In particular, it sends every point in $\left\{-1,1\right\}^8$ to a vector having exactly one entry of size 2 and that is the largest entry.
Kunisky \cite{kun} studies the matrix for a very different reason: if we normalize all the columns of the matrix to have $\ell^2-$norm 1 to arrive at a matrix $B$, then, for every $x \in \left\{-1,1\right\}^8$ we have that $\|B x\|_{\ell^{\infty}} = \sqrt{2} + 1/2 \sim 1.91\dots$ which proves that the Komlos constant is at least $\sqrt{2}+1/2$. It is curious that the same matrix appears to also do exceedingly well for our problem. Of course, it is conceivable that this matrix may not be the best one for our problem when $n=8$.  There is a near-miss with an interesting structure derived from the Hadamard matrix $H_8$, see Fig. 4.

\begin{figure}[h!]
\begin{tikzpicture}
\node at (0,0) {$ B = \frac{1}{\sqrt{7}} \left(
\begin{array}{cccccccc}
 1 & 1 & 1 & 1 & 1 & 1 & 1 & 0 \\
 1 & 1 & 1 & 0 & -1 & -1 & -1 & -1 \\
 1 & 1 & -1 & 0 & -1 & -1 & 1 & 1 \\
 1 & 1 & -1 & -1 & 1 & 1 & -1 & 0 \\
 1 & -1 & -1 & 1 & 1 & -1 & -1 & 0 \\
 1 & -1 & -1 & 0 & -1 & 1 & 1 & -1 \\
 1 & -1 & 1 & 0 & -1 & 1 & -1 & 1 \\
 1 & -1 & 1 & -1 & 1 & -1 & 1 & 0 \\
\end{array}
\right)$};
\node at (6.5,0) {\includegraphics[width=0.3\textwidth]{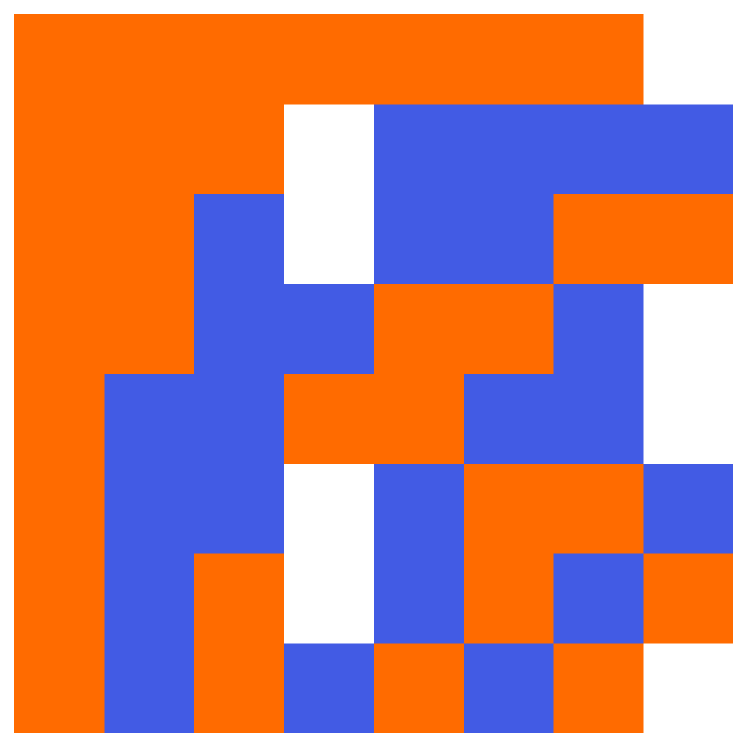}};
\end{tikzpicture}
\caption{Left: a near-optimal matrix for $n=8$. Right: visual representation, each entry is $\pm 1/\sqrt{7}$. $\beta(B) =  3 \sqrt{7}/4 \sim 1.984\dots.$}
\end{figure}

\section{Proofs}
\subsection{Outline.} We start with a quick outline of the arguments. Suppose $a \in \mathbb{R}^n$ is the first row of the matrix $A$ and $\|a\|_{\ell^2}=1$. Hoeffding's inequality can be used to deduce, for $t>0$,
$$ \mathbb{P}\left(\sum_{i=1}^{n}  \varepsilon_i a_{i} \geq t\right) \leq e^{-t^2/2}$$
which shows that the probability of $\left\langle a, x\right\rangle$ exceeding $ \sqrt{(2+\delta) \log{n}}$ is $\leq n^{-\delta/2} n^{-1}$. This is true for each row. The union bound shows that the probability of any of the $n$ rows exceeding $ \sqrt{(2+\delta) \log{n}}$ is $\leq n^{-\delta/2}$. Re-running the same argument with different parameters shows that those few that do indeed exceed that size are unlikely to exceed it by very much and do not significantly alter the average.
As for the lower bound, we will argue that the argument giving the upper bound is very nearly tight when each entry of $A$ is comprised of iid $\pm 1/\sqrt{n}$ entries. Then $\left\langle a, x\right\rangle$ follows a Bernoulli distribution. Tusn\'ady's lemma implies that, for a fixed vector $x_0 \in \left\{-1,1\right\}^n$ and a random matrix $A$ with iid $\pm 1/\sqrt{n}$ entries there are typically many entries of $Ax_0$ that are $\geq \sqrt{(2-\varepsilon) \log{n}}$. Double-counting shows that there exists at least one matrix $A_0$ such that $\|A_0 x\|_{\ell^{\infty}}$ large for most $x$.

\subsection{Proof of the upper bound.}
\begin{proof} Let $a \in \mathbb{R}^n$ be an arbitrary row of the matrix, normalized to $\|a\|_{\ell^2} = 1$. We consider $(Ax)_i = \left\langle a, x \right\rangle$. The proof of the upper bound is fairly straight-forward application of Hoeffding's inequality implying that
$$ \mathbb{P}\left( X_{a} \geq t\right) \leq e^{-t^2/2}$$
independently of $a$. We quickly recall the (standard) argument for the convenience of the reader. The
random variable
$ X = \sum_{i=1}^{n}  \varepsilon_i a_{i}$
satisfies, for all $x,t>0$ 
$$ \mathbb{P}(X \geq x) = \mathbb{P}( e^{tX} \geq e^{tx}) \leq e^{-t x} \cdot \mathbb{E}~ e^{tX}.$$
We have
$$ \mathbb{E} ~e^{t X} = \mathbb{E} \prod_{i=1}^{n} e^{\varepsilon_i t a_i} =  \prod_{i=1}^{n} \mathbb{E} e^{\varepsilon_i t a_i} = \prod_{i=1}^{n} \frac{e^{- ta_i} + e^{ t a_i}}{2} =  \prod_{i=1}^{n} \cosh( t a_i).$$
Using $\cosh{(x)} \leq \exp(x^2/2)$ and the $\ell^2-$normalization of $a$, we have
$$ \mathbb{E} ~e^{t X} \leq \prod_{i=1}^{n} e^{t^2 a_i^2/2} = e^{t^2/2}.$$
Setting $t=x$ and using the symmetry of the random variable
$$ \mathbb{P}\left(\left|\sum_{i=1}^{n}  x_i a_{i} \right| \geq t\right) \leq 2  e^{-t^2/2}.$$
Thus, for any $\varepsilon > 0$, we have
$$ \mathbb{P}\left(\left|\sum_{i=1}^{n}  x_i a_{i} \right| \geq \sqrt{(2+\varepsilon) \log{n}}\right) \leq 2  n^{-1 - \frac{\varepsilon}{2}}.$$
Applying the union bound, we deduce
$$ \mathbb{P}\left( \max_{1 \leq i \leq n} \| Ax\|_{\ell^{\infty}} \geq  \sqrt{(2+\varepsilon) \log{n}}\right) \leq 2  n^{-\varepsilon}.$$
It remains to ensure that the exceptional sets do not contribute a very large expected value. This is easy since
\begin{align*}
\mathbb{E}X &\leq  \sqrt{(2+\varepsilon) \log{n}} +  \sqrt{3 \log{n}} \cdot 2 n^{-\varepsilon} \\
&+ \sum_{k=3}^{\infty} n \cdot \mathbb{P}\left(\sqrt{k \log{n}} \leq \left|\sum_{i=1}^{n}  \varepsilon_i a_{i} \right| \leq \sqrt{(k+1)\log{n}}\right) \sqrt{(k+1) \log{n}} \\
&\leq \sqrt{(2+\varepsilon) \log{n}} + \sum_{k=3}^{\infty} n \cdot \mathbb{P}\left(  \left|\sum_{i=1}^{n}  \varepsilon_i a_{i} \right| \geq \sqrt{k \log{n}}\right)\sqrt{(k+1) \log{n}} + o(1)\\
&\leq  \sqrt{(2+\varepsilon) \log{n}} + \sum_{k=3}^{\infty} n \cdot n^{-k/2}\sqrt{(k+1) \log{n}} + o(1) \\
&\leq   \sqrt{(2+\varepsilon) \log{n}} + o(1).
\end{align*} 
Since $\varepsilon>0$ was arbitrary, we obtain the result.
\end{proof}

\subsection{Proof of the lower bound.}
We need an anti-concentration result that follows quickly from Tusn\'ady's Lemma \cite{bre, mas, tus}.
\begin{lemma} Let $X = \pm 1/\sqrt{n} \pm 1/\sqrt{n} \dots \pm 1/\sqrt{n}$ with each sign being iid. Then, for each $\varepsilon > 0$ there exists constants $\delta, c_{\delta} > 0$ (depending on $\varepsilon$) such that
$$ \mathbb{P}\left( X \geq \sqrt{(2 - \varepsilon) \log{n}}\right) \geq c_{\delta} \frac{n^{\delta}}{n}.$$
\end{lemma}
\begin{proof} The random variable $X = \pm 1/\sqrt{n} \pm 1/\sqrt{n} \dots \pm 1/\sqrt{n}$ is a rescaling of the standard binomial distribution $B(n,1/2)$ and, more precisely,
$$ \frac{\sqrt{n} \cdot X+n}{2} = B(n,1/2).$$
Therefore
$$ \mathbb{P}(X \geq c\sqrt{\log{n}}) = \mathbb{P}\left( B(n,1/2) \geq \frac{n + c \sqrt{n \log{n}}}{2} \right).$$
We use Tusn\'ady's Lemma in the following form (see \cite{bre} or \cite{mas}): if $B(n,1/2)$ denotes the binomial random variable and $Y$ denotes a standard $\mathcal{N}(0,1)$ Gaussian, then, for all integers $0 \leq j \leq n$,
$$ \mathbb{P}\left(B(n,1/2) \geq \frac{n+j}{2} \right) \geq \mathbb{P}\left( Y \geq 2\sqrt{n}\left(1 - \sqrt{1 - \frac{j}{n}}\right)\right).$$
One sees with a Taylor expansion that
$$ 1 - \sqrt{1 - \frac{j}{n}} =  \frac{j}{2n} + \mathcal{O}\left(\frac{j^2}{n^2}\right)$$
Thus
\begin{align*}
 \mathbb{P}(X \geq c\sqrt{\log{n}})  &\geq  \mathbb{P}\left( Y \geq 2\sqrt{n} \left(1 - \sqrt{1 - \frac{j}{n}}\right)\right) \\
 &=  \mathbb{P}\left( Y \geq c \sqrt{\log{n}} + \mathcal{O}\left(\frac{\log{n}}{n^{}}\right)\right).
 \end{align*}
At this point, we use a standard tail bound for the Gaussian
$$ \mathbb{P}\left( Y \geq x\right) \geq \frac{x}{x^2 + 1} \frac{1}{\sqrt{2\pi}} e^{-x^2/2}.$$
Setting $x = \sqrt{(2 - \varepsilon) \log{n}}$ gives the desired result.
\end{proof}

\begin{proof}
We consider all possible $n \times n$ matrices where each entry is $\pm 1/\sqrt{n}$
$$ \mathcal{S} =  \left\{-\frac{1}{\sqrt{n}}, \frac{1}{\sqrt{n}} \right\}^{n \times n}.$$
All rows are normalized in $\ell^2$. 
There are $2^{n^2}$ such matrices and we equip this set with the uniform probability measure. Our goal is to show
the existence of a single matrix $A_0 \in  \mathcal{S} $
such that, for `most' of the points of the unit cube the matrix sends the point to one with at least one large coordinate (where `most' means all but a fraction tending to 0 as $n \rightarrow \infty$). For a constant $c>0$ to be chosen later, we will bound the number of large matrix-point incidences
$$ \# \left\{ (A, x) \in  \mathcal{S}  \times \left\{-1,1\right\}^n: \|Ax\|_{\ell^{\infty}} \geq c \sqrt{\log{n}} \right\}.$$
The main argument will be to show, for a suitable choice of $c$, every single lattice point in the discrete cube $x_0 \in \left\{-1,1\right\}^n$ has the property that all but a vanishing fraction of matrices in $\mathcal{S}$ send it to a point with at least one large entry:
$$\forall~x_0 \in \left\{-1,1\right\}^n \quad \# \left\{ A \in  \mathcal{S} : \|Ax_0\|_{\ell^{\infty}} \geq c\sqrt{\log{n}} \right\} = (1-o(1)) \cdot \# \mathcal{S}. \quad (\diamond)$$
This shows that
$$ \# \left\{ (A, x) \in  \mathcal{S}  \times \left\{-1,1\right\}^n: \|Ax\|_{\ell^{\infty}} \geq c \sqrt{\log{n}} \right\} = (1-o(1)) \cdot \# S \cdot 2^n$$
which then forces the existence of $A_0 \in  \mathcal{S} $ such that
$$ \# \left\{ x \in \left\{-1,1\right\}^n: \|A_0x\|_{\ell^{\infty}} \geq c \sqrt{\log{n}} \right\} = (1-o(1)) \cdot 2^n.$$
Naturally, this shows slightly more: it shows that most random matrices have the desired property.
It remains to establish $(\diamond)$. We fix an abitrary $x_0 \in \left\{-1,1\right\}^n$. As can be seen from the symmetries of the probblem, all $x_0$ will behave identically. Choose a random $A \in \mathcal{S}$ by simply picking each coordinate $\pm 1/\sqrt{n}$ with equal likelihood. All entries and, in particular, all rows are independent of each other.
An arbitrary entry of the random vector $Ax_0$ is then distributed according to a rescaled Binomial distribution since
 $$ (Ax_0)_i = \pm \frac{1}{\sqrt{n}} \pm \frac{1}{\sqrt{n}} \dots \pm \frac{1}{\sqrt{n}}.$$
Using the Lemma, we conclude that
$$ \mathbb{P} \left((Ax_0)_i \geq \sqrt{(2-\varepsilon)\log{n}}\right) \geq c_{\delta} \frac{n^{\delta}}{n}$$
and thus, using independence of the rows of $A$,
$$ \mathbb{P}\left( \|Ax_0\|_{\ell^{\infty}} \leq \sqrt{(2-\varepsilon)\log{n}}\right) \leq \left(1 -c_{\delta} \frac{n^{\delta}}{n} \right)^n \leq e^{- c_{\delta} n^{\delta}}.$$

This shows that for any fixed $x_0 \in \left\{-1,1\right\}^n$ `most' matrices, meaning $1-o(1)$ of all, will have at least one large entry of size $ \sqrt{(2-\varepsilon)\log{n}}$. This proves $(\diamond)$. Since $\varepsilon>0$ can be chosen arbitrarily small, this then implies the result.
\end{proof}

\subsection{Proof of the Proposition}
\begin{proof}
The result follows from adapting a construction of Kunisky \cite{kun} who considers the sequence of matrices
$$ \begin{pmatrix} 1 \end{pmatrix}, \begin{pmatrix} 1 & 1 \\ 1 & -1 \end{pmatrix}, \begin{pmatrix} 1 & 1 & 1 & 0\\ 1 & 1 &-1 &0 \\ 1& - 1 & 0 & 1 \\ 1 & - 1 & 0 &-1 \end{pmatrix}, \left(
\begin{array}{cccccccc}
 1 & 1 & 1 & 0 & 1 & 0 & 0 & 0 \\
 1 & 1 & 1 & 0 & -1 & 0 & 0 & 0 \\
 1 & 1 & -1 & 0 & 0 & 1 & 0 & 0 \\
 1 & 1 & -1 & 0 & 0 & -1 & 0 & 0 \\
 1 & -1 & 0 & 1 & 0 & 0 & 1 & 0 \\
 1 & -1 & 0 & 1 & 0 & 0 & -1 & 0 \\
 1 & -1 & 0 & -1 & 0 & 0 & 0 & 1 \\
 1 & -1 & 0 & -1 & 0 & 0 & 0 & -1 \\
\end{array}
\right)$$
which are perhaps best read after taking a transpose: the transpose corresponds to the discrete wavelet transform of the Haar wavelet.
Up to a reordering of the rows and columns, these matrices can be generated iteratively via
$$ A_{n+1} = \begin{pmatrix} A_n && \mbox{Id}_{2^n \times 2^n} \\ A_n && -  \mbox{Id}_{2^n \times 2^n} \end{pmatrix}.$$
We first observe that $A_n \in \mathbb{R}^{2^n \times 2^n}$. Moreover, one can see by induction that $A_n$ maps the discrete $2^n-$dimensional cube to points with maximal coordinate $n+1$. This is clear when $n=0$. As for the induction, if $x = (x_1, x_2)$, then
$$ A_{n+1} x = \begin{pmatrix} A_n x_1 + x_2 \\ A_n x_1 - x_2 \end{pmatrix}.$$
By induction assumption $\| A_n x_1\|_{\ell^{\infty}}  = n+1$ and thus there exists an entry that is either $(n+1)$ or $-(n+1)$. In either case, we can find a new column whose absolute value is $n+2$. The second observation is that each row of $A_n$ has $n+1$ entries that are $+1$. Thus, taking $B_n = A_n/\sqrt{n+1}$ we obtain a matrix whose columns are all normalized in $\ell^2$ and for which
$$ \frac{1}{2^{2^n}} \sum_{x \in \left\{-1,1\right\}^{2^n}} \|B_n x\|_{\ell^{\infty}} = \sqrt{n+1}.$$

\end{proof}


\begin{thebibliography}{10}

\bibitem{alon}  N. Alon and Z. F\"uredi, Covering the cube by affine hyperplanes. European journal of combinatorics, 14 (1993), p. 79--83.

\bibitem{bana2}  W. Banaszczyk, Balancing vectors and convex bodies, Studia Math. 106 (1993), 93--100.

\bibitem{bana3}  W. Banaszczyk, Balancing vectors and gaussian measures of n-dimensional convex bodies, Random Structures \& Algorithms, 12(4):351--360, 1998.

\bibitem{bana}  W. Banaszczyk and S.J. Szarek, Lattice coverings and Gaussian measures of n-dimensional convex bodies, Discrete Comput. Geom. 17 (1997), p. 283--286.

\bibitem{bre} J. Bretagnolle and P. Massart, Hungarian Constructions from the Nonasymptotic Viewpoint, Ann. Probab. 17 (1989), p. 239-256

\bibitem{haj} D. Hajela, On a conjecture of Komlos about signed sums of vectors inside the sphere. European Journal of Combinatorics, 9(1), 33--37.

\bibitem{henk} M. Henk and G. Tsintsifas, Lattice point coverings,
Adv. Math. (China) 36 (2007), no. 4, 441--446.


\bibitem{kun} D. Kunisky, The discrepancy of unsatisfiable matrices and a lower bound for the Koml\'os conjecture constant. SIAM Journal on Discrete Mathematics, 37 (2023), p. 586--603.

\bibitem{li} L. Li and A. Nikolov, On the Gap between Hereditary Discrepancy and the Determinant Lower Bound. arXiv preprint arXiv:2303.08167.

\bibitem{mac} F. J. MacWilliams and N. J. A. Sloane, The Theory of Error-Correcting Codes. Vol. 16 (1981). North-Holland.

\bibitem{mas} P. Massart,  Tusnady's lemma, 24 years later. In Annales de l'IHP Probabilites et statistiques 38 (2002), pp. 991-1007.

\bibitem{reis} V. Reis and T. Rothvoss,  Vector balancing in Lebesgue spaces. Random Structures \& Algorithms, 62 (2023), p. 667--688.

\bibitem{spencer} J. Spencer, Six standard deviations suffice, Transactions of the American mathematical society, 289 (1985), p.679-706.

\bibitem{tus} G. Tusn\'ady, A study of statistical hypotheses, PhD Thesis, Hungarian Academy of Sciences, Budapest, 1977 (In Hungarian).

\bibitem{vershynin} R. Vershynin, High-dimensional probability: An introduction with applications in data science (Vol. 47). Cambridge University Press, 2018.


\end{thebibliography}
\end{document}